\documentclass[11pt]{article} 

\usepackage{amsmath,amsfonts,amssymb,amsthm}
\usepackage{mathtools}
\usepackage{graphicx}
\usepackage{epsfig}
\usepackage{graphicx}

\newcommand{\letitre}{A Riemann--Hilbert correspondence for infinity local systems}

\usepackage[margin=2cm,includeheadfoot]{geometry}

\usepackage{fancyhdr}
\fancypagestyle{plain}{%
  \fancyhead[LO]{}
  \fancyhead[CO]{}\fancyhead[RO]{}
}

\newcommand{\A}{\mathcal{A}}
\newcommand{\E}{\mathbb{E}}
\newcommand{\R}{\mathbb{R}}

\newcommand{\Z}{\mathbb{Z}}

\newcommand{\Reps}{\mathrm{Reps}}
\newcommand{\Hom}{\mathrm{Hom}}
\newcommand{\End}{\mathrm{End}}
\newcommand{\Sing}{\mathrm{Sing}}
\newcommand{\Flat}{\mathrm{Flat}}
\newcommand{\Ob}{\mathcal{O}b\,}
\newcommand{\ev}{\mathrm{ev}}

\newcommand{\mto}[1]{\stackrel{#1}{\longrightarrow}}
\newcommand{\As}{\text{\bfseries\sf{A}}}
\newcommand{\Perf}{\mathcal{P_{\As}}}
\newcommand{\Pinf}{{\pi_{\infty}}M}
\DeclareMathOperator{\Ho}{Ho}

\newcommand{\Cal}{\mathcal{C}}

\newcommand{\Loc}{\mathsf{Loc}_{\infty}^{\Cal}(K)}
\newcommand{\ALoc}{\mathsf{Loc}_{A_{\infty}}(K)}

\newcommand{\RH}{\mathcal{RH}}
\DeclareMathOperator{\I}{Im}
\newtheorem{thm}{Theorem}[section]
\newtheorem{Defn}[thm]{Definition}
\newtheorem{Prop}[thm]{Proposition}
\newtheorem{Thm}[thm]{Theorem}
\newtheorem{Rmk}[thm]{Remark}
\newtheorem{Cor}[thm]{Corollary}
\newtheorem{Lem}[thm]{Lemma}
\newtheorem{Example}[thm]{Example}

\newcommand{\sSet}{s\mathcal{S}et}

\numberwithin{equation}{section}

\begin{document}

\title{A Riemann--Hilbert correspondence for infinity local systems}

\author{Jonathan Block and Aaron M. Smith}




\maketitle
\begin{abstract}
We describe an $A_\infty$-quasi-equivalence of dg-categories between the first authors' $\mathcal{P}_{\mathcal{A}}$ ---the category of category of prefect $A^0$-modules with flat $\Z$-connection, corresponding to the de Rham dga $\mathcal{A}$ of a compact manifold $M$--- and the dg-category of \emph{infinity-local systems} on $M$ ---homotopy coherent representations of the smooth singular simplicial set of $M$, $\Pinf$. We understand this as a generalization of the Riemann--Hilbert correspondence to $\Z$-connections ($\Z$-graded superconnections in some circles).  In one formulation an infinity-local system is simplicial map between the simplicial sets ${\pi}_{\infty}M$ and a repackaging of the dg-category of cochain complexes by virtue of the simplicial nerve and Dold-Kan.  This theory makes crucial use of Igusa's notion of higher holonomy transport for $\Z$-connections which is a derivative of Chen's main idea of generalized holonomy.
\end{abstract}

\section{Introduction} \label{sec:1}
Given a compact manifold $M$,  the classical Riemann--Hilbert correspondence gives an equivalence of categories between $\Reps(\pi_1(M))$ and the category\newline $\Flat(M)$ of vector bundles with flat connection on $M$.  While beautiful, this correspondence has the primary drawback that it concerns the truncated object $\pi_1$, which in most cases contains only a small part of the data which comprises the homotopy type of $M$.  From the perspective of (smooth) homotopy theory  the manifold $M$ can be replaced by its infinity-groupoid $\pi_{\infty}M := \Sing_{\bullet}^{\infty}M $ of smooth simplices.  Considering the correct notion of a representation of this object will allow us to produce an untruncated Riemann--Hilbert theory.  More specifically, we define an \emph{infinity-local system} to be a map of simplicial sets which to each simplex of $\pi_{\infty}M$ assigns a homotopy coherence in the category of chain complexes over $\R, \Cal := \mathrm{Ch}(\R)$.  Our main theorem is an $A_\infty$-quasi-
equivalence
\begin{equation*}
 \RH \colon  \Perf \rightarrow \mathsf{Loc}^{\Cal}(\pi_{\infty}M),
\end{equation*}
in which $\Perf$ is the dg-category of graded bundles on $M$ with flat $\Z$-graded connection, and $\mathsf{Loc}^{\Cal}(\pi_{\infty}M)$ is the dg-category of infinity-local systems on $M$.

In the classical Riemann--Hilbert equivalence, the map
\begin{equation*}
 \Flat(M) \rightarrow \Reps(\pi_1(M)),
\end{equation*}
is developed by calculating the holonomy of a flat connection.  The holonomy descends to a representation of $\pi_1(M)$ as a result of the flatness.  The other direction, 
\begin{equation*}
 \Reps(\pi_1(M)) \rightarrow \Flat(M),
\end{equation*}
is achieved by the associated bundle construction.

In the first case our correspondence proceeds analogously by a calculation of the holonomy of a flat $\Z$-graded connection.  The technology of iterated integrals suggests a precise and rather natural notion of such holonomy.  Given a vector bundle $V$ over $M$ with connection, the usual parallel transport can be understood as a form of degree $0$ on the path space  $PM$  taking values in the bundle $\Hom(\ev_1^*V,\ev_0^*V)$.  The higher holonomy is then a string of forms of total degree $0$ on the path space of $M$ taking values in the same bundle.  Such a form can be integrated over cycles in  $PM$ , and the flatness of the connection implies that such a pairing induces an infinity-local system as desired.  This is the functor
\begin{equation*}
 \RH\colon  \Perf \rightarrow \mathsf{Loc}^{\Cal}(\pi_{\infty}M).
\end{equation*}

It would be an interesting problem in its own right to define an inverse functor which makes use of a kind of associated bundle construction.  However we chose instead to prove quasi-essential surjectivity of $\RH$.  Given an infinity-local system $(F,f)$ one can form a  complex of sheaves over $X$ by considering the sheaf $\mathsf{Loc}^{\Cal}(\pi_{\infty}U)(\R,F)$.  This complex is quasi-isomorphic to the sheaf obtained by extending by the sheaf of $C^{\infty}$ functions and then tensoring with the de Rham sheaf.  Making use of a theorem of Illusie, we construct from this data a perfect complex of $\A^0$-modules quasi-isomorphic to the zero-component of the connection in $\RH(F)$.  Finally we follow an argument of \cite{MR2648899} to construct an element of $\Perf$ which is quasi isomorphic to $\RH(F)$.


\section{Infinity-Local Systems}\label{sec:2}
\subsection{The Definition of an Infinity-Local System}\label{sec:2.1}
Now we develop a higher version of a local system.  These objects will be almost the same as the $A_{\infty}$-functors of Igusa in \cite{MR2846735}, but tailored to suit our equivalence result.  We want to emphasize the analogy with classical local systems.  Let $\mathcal{C}$ be a dg-category over $k$ a characteristic $0$ field (which we are implicitly regarding as $\R$ in this paper), and $K$ a simplicial set.  Fix a map $F \colon  K_0 \rightarrow \Ob \, \mathcal{C}$.  Then define:
\begin{equation*}
 \mathcal{C}_{F}^{i,j} := \lbrace \mbox{maps } f\colon  K_i \rightarrow \mathcal{C}^j \vert\,\, f(\sigma)\in \mathcal{C}^j(F(\sigma_{(i)}),F(\sigma_{(0)}))\rbrace,
\end{equation*}
and,
\begin{equation*}
 \mathcal{C}^k_{F}(K) := \underset{i+j=k,i\geq 0}{\oplus}\mathcal{C}_{F}^{i,j}.
\end{equation*}

There are some obvious gradings to keep track of.  For $f \in \mathcal{C}_{F}^{p,q}$ define
\begin{equation*}
T(f) := (-1)^{\lvert f \rvert}f := (-1)^{p+q}f,  \, \,\,\,
K(f) := (-1)^{q}f,\, \,\,\, J(f) := (-1)^{p}f.
\end{equation*}
With respect to the simplicial degree in $ \mathcal{C}^k_{F}(K)$, we write
\begin{equation*}
 f = f^1 + f^2+\dotsc, \,\,\,\,\,\,f^i \in \mathcal{C}_{F}^{i,\bullet}.
\end{equation*}  

\vspace{.1in}\noindent We define some operations on these maps: 
\begin{align*}
 (df^i)(\sigma_i) &:= d(f^i(\sigma_i))\\
 (\delta f^i)(\sigma_{i+1}) &:= \sum_{l=1}^{i}(-1)^{l} f^i(\partial_l (\sigma))\\ 
 \hat{\delta} &:= \delta \circ T(\bullet)
\end{align*}
and for $g^p \in \mathcal{C}_{F}^{p,q}$,
\begin{equation*}
 (f^i \cup g^p)(\sigma \in K_{i+p}) := (-1)^{i(p+q)}f^i(\sigma_{(0 \dotsc i)})g^p(\sigma_{(i \dotsc p+i)}).
\end{equation*}
These operations can be extended by linearity to sums in $\oplus \mathcal{C}^s(K)$.  The cup product is defined as the sum of the cups across all internal pairs of faces,
\begin{equation*}
 (f \cup g)(\sigma_k) := \sum_{t=1}^{k-1}(-1)^{t\lvert g^{k-t} \rvert}f^t(\sigma_{(0 \dotsc t)})g^{k-t}(\sigma_{(t \dotsc k)}).
\end{equation*}
We could suggestively write $f \cup g:=  \mu \circ (f \otimes g) \circ \Delta$ in which $\mu$ is the composition $\mathcal{C}^{\bullet} \otimes \mathcal{C}^{\bullet} \rightarrow \mathcal{C}^{\bullet}$, and $\Delta$ is the usual comultiplication which splits a simplex into a sum over all possible splittings into two faces,
\begin{equation*}
\Delta(\sigma_k) := \sum_{p+q=k,p,q \geq 1} \sigma_p\otimes \sigma_q.
\end{equation*}
However, strictly speaking there is no $\Delta$ operator on a general simplicial set because one doesn't have a linear structure.  The sign above appears because an $i$-simplex passes an element of total degree $p+q$ ---consistent with the Koszul conventions. 

\begin{Defn}
 A pair $(F,f)$ with $f \in \mathcal{C}_{F}^1(K)$  such that $0 = \hat{\delta} f + df + f \cup f$ is called an infinity-local system.  The set of infinity-local systems valued in $\Cal$ is denoted $\mathsf{Loc}_{\infty}^{\Cal}(K)$  
\end{Defn}
We will often denote an infinity-local system $(F,f)$ by just $F$ if no confusion will arise.


\begin{Example} If $F$ denotes an ordinary local system, then it naturally defines an infinity-local system.
\end{Example}
\begin{proof}
 Exercise.
\end{proof}

\begin{Rmk}
In the category of cochain complexes, $\Cal$, the differentials on arbitrary hom-complexes will be given by graded commutation with the family of differentials  $d_x \in \Cal^1(x,x)$ i.e.,
\begin{equation*}
 df(\sigma_k) = d_{F(\sigma_{(0)})}\circ f_k(\sigma_{(0 \dotsc k)}) - (-1)^{\lvert f_k \rvert}f_k(\sigma_{(0\dotsc k)})\circ d_{F(\sigma_{(k)})}.
\end{equation*}
\end{Rmk}

\subsection{Infinity-Local Systems as Simplicial Maps}\label{sec:2.2}
We can give an alternate, more concise description of an infinity-local system as a certain map of simplicial sets.  In service of this redescription we introduce the notion of the simplicial set of homotopy coherent simplices in a dg-category.   
\begin{Defn}\label{Cinfty}
Given a dg-category $\Cal$, the simplicial set $\Cal_{\infty}$ of homotopy-coherent simplices in $\Cal$ is constructed as follows.

Denote by $Y_i([n])$ the set of length-$i$ (ordered) subsets of $[n]$.  We denote an element of $Y_j([n])$ as an ordered tuple $(i_0 < i_1 < \dotsc < i_{j-1})$ and make use of this notation below.
\begin{flalign*}
    \Cal_{\infty \,0} :=\,& \lbrace P = (p,P_0) \vert \\
		  & \quad p \colon  Y_1([0]) \rightarrow \Ob\Cal, \,\, P_0 =d_p \colon  Y_1([0]) \rightarrow \Cal^{1}(p(i_0),p(i_0))\rbrace \\
		  & \\
  \Cal_{\infty \, 1} := \,& \lbrace P = (p,P_0 + P_1) \vert\\
		  & \quad p \colon  Y_1([1]) \rightarrow \Ob\Cal, \,\, P_0= d_p \colon  Y_1([1]) \rightarrow \Cal^{1}(p(i_0),p(i_0)),\\
                  & \quad P_1\colon  Y_2([1]) \rightarrow \mathcal{C}^0(p(i_1),p(i_0)), \\
		  & \text{such that,}\\
		  & \quad dP_1 = 0 \rbrace\\
  & \\
  & \dotsc \\
  \Cal_{\infty \, l}  := \,& \lbrace P = (p,\sum_{i=0}^l P_i) \vert\\
                 & \quad p\colon  Y_1([l]) \rightarrow \Ob\Cal, \,\, P_0 = d_p \colon  Y_1([l]) \rightarrow \Cal^{1}(p(i_0),p(i_0)),\\
                 & \quad P_1 \colon  Y_{2}([l]) \rightarrow  \mathcal{C}^{0}(p(i_1),p(i_0)),\\
		  & \dotsc\\
                 & \quad P_j \colon  Y_{j+1}([l]) \rightarrow  \mathcal{C}^{1-j}(p(i_j),p(i_0)),\\
                 & \dotsc\\
                 & \quad P_l \colon  Y_{l+1}([l]) \rightarrow \mathcal{C}^{1-l}(p(i_l),p(i_0)), \\
		 & \text{such that,}\\
		 & \quad dP + \hat{\delta} P + P \cup P = 0 \rbrace.
\end{flalign*}
The equation $dP + \hat{\delta} P + P \cup P = 0$ above can be parsed according to the following three definitions:
\begin{align*}
 dP(i_1 < \dotsc < i_j) &:= d(P(i_1 < \dotsc < i_j))\\
 \hat{\delta} P_j(i_0 < \dotsc < i_{j+1}) &:= -\sum_{q=1}^{k-1} (-1)^q P_j(i_0 < \dotsc <  \hat{i_q} < \dotsc < i_{j+1})\\
 (P \cup P)_j(i_0 < \dotsc < i_{j}) &:= \sum_{q=1}^{j-1} (-1)^{q}P_q(i_0 < \dotsc < i_q)\circ P_{j-q}(i_q < \dotsc < i_{j}).
\end{align*}
The face and degeneracy maps for this simplicial set are defined as follows.
\begin{align*}
 \partial_q P_j(i_0 < \dotsc < i_{j+1}) &:= P_j(i_0 < \dotsc < \hat{i_q} < \dotsc < i_{j+1}),\\
 s_q P_j(i_0 < \dotsc < i_{j-1}) & := P_j(i_0 < \dotsc < i_q < i_q < \dotsc < i_{j-1}).
\end{align*}\\
\end{Defn}

\begin{Rmk}
This explicit definition of $\Cal_{\infty}$ can be compressed to the definition  
\begin{equation}
 \Cal_{\infty} :=  N \circ S \circ \tilde{\Gamma} \circ T[\Cal].
\end{equation}
Here $T$ is the functor which truncates the hom-complexes of a dg-category to connective complexes by taking homology at the $0$th grading.  $\tilde{\Gamma}$ is the functor produced by application of the Dold-Kan equivalence to hom-complexes of a connective dg-category.  $S$ is the forgetful functor mapping a category enriched in simplicial vector spaces to a category enriched in simplicial sets.  Finally, $N$ is the simplicial nerve due to Cordier \cite{MR648798} and described in \cite{MR2522659}.
\end{Rmk}

Now it is possible to give a more concise description of an infinity-local system as a map of simplicial sets.
\begin{Defn}(alternate)
 An infinity-local system on $K$ valued in $\Cal$ is an element of $\sSet(K,\Cal_\infty)$, i.e. the set of $\sSet$-maps from $K$ to $\Cal_\infty$.  
\end{Defn}
This redescription of an infinity-local system makes it clear that the functor 
\begin{equation*}
K \mapsto \Ob \Loc
\end{equation*}
is represented by $\Cal_\infty$.  

Going even further, this perspective inspires a potential definition of the category (or space) of such objects as a mapping space in the enriched setting, but this story will be left for another exploration.  In the next section we describe the category of infinity-local systems explicitly as a dg-category.

\subsection{The dg-category of Infinity-Local Systems}\label{sec:2.3}
\begin{Defn}
 We denote by $\mathsf{Loc}_{\infty}^\Cal(K)$, the category of infinity-local systems.  The objects are infinity-local systems on the simplicial set $K$ valued in the dg-category $\Cal$.  We define a complex of morphisms between two infinity-local systems  $F$,$\,G$:
\begin{equation*}
  \mathsf{Loc}_{\infty}^{\Cal}(K)(F,G) := \underset{i+j=k}{\oplus} \lbrace \phi\colon  K_i \rightarrow \Cal^j \vert \phi(\sigma) \in \Cal^j(F(\sigma_{(i)}),G(\sigma_{(0)})) \rbrace,
\end{equation*}
with a differential D,
\begin{equation*}
 D\phi := \hat{\delta}\phi +d\phi + G \cup \phi - (-1)^{\lvert \phi \rvert}\phi \cup F.
\end{equation*}
In the above, $\phi = \phi^0 + \phi^1 +\dotsc$ is of total degree $\lvert \phi \rvert =p$, and
\begin{align*}
 (\hat{\delta} \phi)(\sigma_k) &:= \delta \circ T = \sum_{j=1}^{k-1}(-1)^{j+\lvert \phi \rvert}\phi^{k-1}(\partial_j(\sigma_k))\\
 d\phi(\sigma_k) &:= d_{G(\sigma_{(0)})} \circ \phi(\sigma_{(0 \dotsc k)}) - \phi(\sigma_{(0 \dotsc k)})\circ d_{F(\sigma_{(k)})}. 
\end{align*}  
\end{Defn}

\begin{Prop}
 $\mathsf{Loc}_{\infty}^{\Cal}(K)$ is a dg-category.
\end{Prop}
\begin{proof}
$D^2=0$ follows from the two observations,
\begin{align*}
 i)&\,\,\, \hat{\delta}[F, G] = [\hat{\delta} F, G] + (-1)^{\lvert F \rvert}[F,\hat{\delta} G],\\
 ii)&\,\,\, [F,[G,H]] = [[F,G],H]+ (-1)^{\lvert F \rvert \lvert G \rvert}[G,[F,H]],
\end{align*}
---in which $[\,,]$ is the graded commutator $[A,B] = A \cup B + (-1)^{\lvert A \rvert \lvert B \rvert}B \cup A$--- and the fact $F$ and $G$ are infinity-local systems:
\begin{equation*}
 \hat{\delta} F +dF + F \cup F = \hat{\delta} G + dG + G \cup G = 0.
\end{equation*}
\end{proof}

We can define a shift functor in $\Loc$ as well as a cone construction.  
\begin{Defn}\label{ShiftAndCone}
Given $F \in \Loc$, define $F[q]$ via 
\begin{equation*}
F[q](x\in K_0) := F(x)[q], \,\,\ \text{and} \,\,\, F[q](\sigma_k) := (-1)^{q(k-1)}F(\sigma_k).
\end{equation*}
For a morphism,
\begin{equation*}
  \phi[q](\sigma_k) := (-1)^{qk}\phi.
\end{equation*}
Given a morphism $\phi \in \Loc(F,G)$ of total degree $q$, define the map
\begin{equation*}
 C(\phi)\colon  K_0 \rightarrow \Ob \mathcal{C}\,\,\,\, \text{by} \,\,\,\, x \mapsto F[1-q](x)\oplus G(x). 
\end{equation*}
Define the element $c(\phi)$ of $\mathcal{C}^1_{C(\phi)}$ via
\begin{equation*}
c(\phi) := \begin{pmatrix}
 f[1-q]  & 0 \\
 \phi[1-q] & g
\end{pmatrix}.
\end{equation*}
\end{Defn}
Unless $\phi$ is closed, this cone will not be an element of $\Loc$. This useful construction will appear in our calculations later both when $\phi$ is closed and otherwise.

\begin{Defn}
 A degree $0$ closed morphism $\phi$ between two infinity-local systems $F,\,G$
over $K$ is a homotopy equivalence if it induces an isomorphism in $\Ho\mathsf{Loc}_{\infty}^{\Cal}(K)$.
\end{Defn}
We want to give a simple criterion for $\phi$ to define such a homotopy equivalence.  On the complex $\mathsf{Loc}_{\infty}^{\Cal}(K)^\bullet (F,G)$
define a decreasing filtration by
\[ F^k \mathsf{Loc}_{\infty}^{\Cal}(K)^\bullet(F,G) =\{\phi\in
 \mathsf{Loc}_{\infty}^{\Cal}(K)^\bullet(F,G) )| \,\,\,\phi^i=0 \mbox{ for }
i<k\}.\] 
\begin{Prop}\label{SS} There is a spectral sequence
\begin{equation*} E_0^{pq}\Rightarrow
H^{p+q}( \mathsf{Loc}_{\infty}^{\Cal}(K)^\bullet(F,G)), 
\end{equation*}
in which
\[ E_0^{pq}= \mbox{gr } (\mathsf{Loc}_{\infty}^{\Cal}(K)^\bullet(F,G) )= \lbrace \phi\colon  K_p \rightarrow \Cal^q \vert \phi(\sigma) \in \Cal^q(F(\sigma_{(i)}),G(\sigma_{(0)})) \rbrace,
\] with differential
\[d_0(\phi^p)=d_G\circ\phi^p-(-1)^{p+q}\phi^p\circ d_F.\]
\end{Prop}

\begin{Cor}\label{perfect} For two infinity-local systems $F$ and $G$, the $E_1$-term of the spectral sequence is a local system in the ordinary sense. 
\end{Cor}

\begin{Prop}\label{HE}
A closed morphism $\phi \in \mathsf{Loc}_{\infty}^{\Cal}(K)^0(F,G)$ is a homotopy equivalence if and only if $\phi^0\colon (F_x, d_F)\to (G_x,d_G)$ is a
quasi-isomorphism of complexes for all $x\in K_0$.
\end{Prop}
\begin{proof} The proof follows as in the proof of Proposition 2.5.2 in \cite{MR2648899}.\end{proof}

\subsection{$A_\infty$-local systems}\label{sec:2.4}
In what remains of section $2$ we present a mild extension of the notion of an infinity-local system.  This notion is included for the sake of general interest and potential future applications.  The basic observation is that an infinity-local system need not strictly take values in a dg-category; there are many variations on the main theme.  As an example, in this section we present a notion of an $A_\infty$-local system ---an infinity-local system valued in an $A_\infty$-category.  We will use almost entirely the same notation as before.

Let $\mathcal{C}$ be an $A_{\infty}$-category, with multiplications denoted $\mu_i$, and let $K$ be a simplicial set.  

As before, an object $F$ consists of a choice of a map $F \colon K_0 \rightarrow \Ob\mathcal{C}$ along with an element $f$ of total degree $1$ from the set
\begin{equation*}
 f \in \mathcal{C}^1_{F}(K) := \underset{i+j=1,i\geq 0}{\oplus}\mathcal{C}_{F}^{i,j},
\end{equation*}
with,
\begin{equation*}
 \mathcal{C}_{F}^{i,j} := \lbrace \mbox{$k$-linear maps } f\colon K_i \rightarrow \mathcal{C}^j \vert F(\sigma)\in \mathcal{C}^j(F(\sigma_{(i)}),F(\sigma_{(0)}))\rbrace.
\end{equation*}
$f$ will be required to satisfy a generalized Maurer--Cartan equation.

Morphisms are also as before: 
\begin{equation*}
 \ALoc^q(F,G) = \lbrace \mbox{k-linear maps} K_i \rightarrow \mathcal{C}^j(F(i),G(0))\rbrace.
\end{equation*}

We define a series of multiplications on composable tuples of morphisms.  Consider an $n+1$-tuple of objects $(F_n,\dotsc, F_0)$ and a corresponding tuple of composable morphisms $(\phi_n \otimes \dotsc \otimes \phi_0)$.  Then we have
\begin{equation*}
 m_n \colon \otimes_{n\geq i \geq 0}\ALoc^{\bullet}(F_{i+1},F_i) \rightarrow \ALoc(F_0,F_n)^{\bullet}[2-n],
\end{equation*}
given by, for $n=1$,
\begin{equation*}
 m_1 \colon \phi \mapsto \mu_1\circ (\phi) -(-1)^{\lvert \phi \rvert}\phi \circ \mu_1+ (-1)^{\lvert \phi \rvert}(\sum_l (-1)^l \phi\circ \partial_l),
\end{equation*}
and for $n \geq 1$,
\begin{equation*}
 m_n\colon (\phi_n \otimes \dotsc \otimes \phi_0) \mapsto \mu_n \circ (\phi_n \otimes \dotsc \otimes \phi_0) \circ \Delta^{(n)}.
\end{equation*}

\begin{Defn}
A pair $(F,f)$ with $f \in \mathcal{C}_{F}^1(K)$ such that $0 = \sum_{i=1}^{\infty} m_i(f^{\otimes i})$ is called an $A_\infty$-local system.  The set of $A_\infty$-local systems is denoted $\mathsf{Loc}_{A_\infty}^{\Cal}(K)$.
\end{Defn}

It is important to note that the Maurer--Cartan equation above is not finite, but has a finite number of terms when evaluated on any simplex due to the fact that $\Delta^{n}(\sigma)=0$ for $n>>0$.

%

\section{Iterated Integrals and Holonomy of $\Z$-graded Connections}\label{sec:3}
Now let $\As=(\mathcal{A}^\bullet(M),d)$ be the de Rham differential graded algebra (DGA) of a compact, closed $C^{\infty}$-manifold  $M$.  
\begin{Defn}
 ${\pi_{\infty}}M$ ---the $\infty$-\emph{groupoid} of $M$--- is $\Sing_{\bullet}^{\infty}M$, the simplicial set over $k = \R$ of $C^{\infty}$-simplices.
\end{Defn}
By $\Cal$ we denote the dg-category of (cohomological) complexes over $\R$.  Our main goal in this text is to derive an $A_\infty$-quasi-equivalence between $\Perf$ and $\mathsf{Loc}_{\infty}^{\Cal}({\pi_{\infty}}M)$.  The former is a dg-category of \emph{modules with superconnection} or \emph{cohesive modules} \cite{MR2648899}; to wit, an object of $\Perf$ is a pair $(E^{\bullet},\E)$ where $E^{\bullet}$ is a $\Z$-graded (bounded), finitely-generated, projective, right $\A^0$-module and $\E$ is a $\Z$-connection with the flatness condition $\E \circ \E = 0$.  This category should not be confused with the category of dg-modules over the one-object dg-category $\mathcal{A}$.  $\Perf$ is a finer invariant ---see \cite{MR2648899}.

By the Serre-Swan correspondence, an object of $\Perf$ corresponds to the smooth sections of a $\Z$-graded vector bundle $V^{\bullet}$ over $M$ with the given flat $\Z$-connection.  In the preprint (arXiv:0912.0249v1) Kiyoshi Igusa presents from scratch a notion of higher parallel transport for a $\Z$-connection.  This is a tweaked example of Chen's higher transport outlined in \cite{MR0454968} which makes crucial use of his theory of iterated integrals.  In this section we reformulate and extend this idea to produce a higher holonomy functor from $\Perf$ to $\Loc$.  To start we present a version of iterated integrals valued in a graded endomorphism bundle.  


\subsection{Sign Conventions}\label{sec:3.1}
Many of the computations in this paper involve manipulating elements in tensor-products of graded vector spaces.  As a result there are a number of relevant gradings and sign-convention choices.  We take the standard approach and employ the Koszul conventions in our computations.  We also use some non-standard notation to denote alternation with respect to various indices, or simply to pick out those indices.

Let $V$ be a graded vector bundle on $M$; then $\End(V)$ is a graded algebra bundle on $M$.  The symbols $T,J,K$ will be used to denote an alternating sign of the degree of a form valued in a graded bundle with respect to the total degree, form degree, and bundle-grading degree respectively.  For instance if $\omega \in V^{q}\otimes_{\A^0}\A^p$, then $T\omega = (-1)^{p+q}\omega, K\omega = (-1)^q\omega$, and $J\omega = (-1)^p\omega$.  The similar convention carries over for forms valued in the $\End(V)$ which has an obvious grading.  Also, an element of $f \in \End^k(V)$ can be broken into a sum $f = \sum_i f_i$ where $f_i \colon  V^i \rightarrow V^{i + k}$.  If we want to pick out this index, we write $I(f_i) := i.$

\subsection{Path Space Calculus}\label{sec:3.2}
In \cite{MR0454968} and earlier works, Chen defined a notion of a differentiable space ---the archetypal differentiable space being  $PM$ for some smooth manifold $M$.  This is a space whose topological structure is defined in terms of an atlas of plots ---maps of convex neighborhoods of the origin in $\R^n$ into the space which cohere with composition by smooth maps--- and the relevant analytic and topological constructs are defined in terms of how they pull back onto the plots.   In particular one can construct a reasonable definition of vector bundles over a differentiable space as well as differential forms.  One can likewise define an exterior differential, and subsequently a so-called Chen de Rham complex \cite{MR1945356}.  We will try to make transparent use of these constructions, but we defer the reader to the detailed discussions of these matters in \cite{MR0454968},\cite{MR1945356}.  

The primary reason that path-space calculus is relevant to our discussion is that the holonomy of a $\Z$-graded connection on $V$ can be defined as a sequence of smooth forms on $PM$  with values in the bundle $\Hom(p_1^*V,p_0^*V)$.  The usual parallel transport will be the $0$-form part of the holonomy.  The higher terms will constitute the so-called \emph{higher holonomy}.

\subsection{Iterated Integrals}\label{sec:3.3}

Let us parametrize the $k$-simplex by $k$-tuples $t=(1 \geq t_1 \geq t_2 \geq \dotsc \geq t_k \geq 0)$.  Then we define the obvious evaluation and projection maps:
\begin{align*}
 \ev_k \colon & PM \times \Delta^k \rightarrow M^k\colon  (\gamma,(t_1,\dotsc, t_k)) \mapsto (\gamma(t_1),\gamma(t_2),\dotsc,\gamma(t_k)),\\
 \pi \colon & PM \times \Delta^k \rightarrow PM \colon (\gamma,(t_1,\dotsc, t_k)) \mapsto \gamma.
\end{align*}
Let $V$ be a graded bundle on $PM$, and in a trivializing patch we identify $\End(V)$ as a graded matrix bundle $E:= Mat^{\bullet}(V)$.    
Define $\iota$ to be the embedding 
\begin{equation*}
 \iota\colon (\Gamma(E\otimes \Lambda^\bullet T^*X))^{\otimes_\R k} \rightarrow \Gamma (E^{\boxtimes k}\otimes \Lambda^\bullet T^*X^{\boxtimes k})
\end{equation*}
\begin{equation*}
(a_1 \otimes_\R \dotsc \otimes_\R a_k) \mapsto (a_1 \boxtimes \dotsc \boxtimes a_k).   
\end{equation*}
Given the space of forms $\ev_k^*(\Gamma E^{\boxtimes k}\otimes \Lambda^\bullet T^*(X)^{\boxtimes k})$ we can use the multiplication in the fibers of $E$ and $\Lambda^\bullet T^*(X)$ to define 
\begin{equation*}
 \mu\colon \ev_k^*\Gamma(E^{\boxtimes k}\otimes (\Lambda^\bullet T^*X)^{\boxtimes k}) \rightarrow \Gamma(p_0^*E\otimes \Lambda^\bullet T^*(PM \times \Delta^k)).
\end{equation*}  
\begin{Defn}
The iterated integral map is the composition
\begin{equation}
 \int a_1 a_2 \dotsc a_k := (-1)^{\spadesuit}\pi_*(\mu ( \ev_k^* ( \iota(a_1 \otimes_\R \dotsc \otimes_\R a_k)))),
\end{equation}
with,
\begin{equation}  \spadesuit = \sum_{1 \leq i< k} (T(a_i)-1)(k-i). 
\end{equation}
Since  $E$  is graded, the elements $\{ a_i \}$ are bi-graded as usual, with $T(\bullet)$ denoting the total degree.
\end{Defn}

\subsection{$\Z$-graded Connection Holonomy}\label{sec:3.4}

Suppose $V$ has a $\Z$-connection $\E$.  Locally $\E$ is of the form $d -[A^0+A^1+\dotsc +A^m]$.  (With the above conventions, $(-1)^kd$ is locally the trivial connection on $E^k$)  Let $\omega = A^0+A^1+\dotsc+A^m$.  This is a form of total degree $1$, i.e. in $\oplus \End^{1-i}(V)\otimes_{\A^0}\A^i$.  To any such form we can associate its holonomy
\begin{equation*}
 \Psi := I + \int \omega + \int \omega \omega + \int \omega \omega \omega +\dotsc, 
\end{equation*}
which breaks further into its components with respect to the form-grading.  For instance,
\begin{equation*}
 \Psi_k =  \int A^{k+1} + \sum_{i+j = k+2} \int A^iA^j + \sum_{i_1+i_2+i_3 = k+3} \int A^{i_1}A^{i_2}A^{i_3} + \dotsc
\end{equation*}

Chen calculated the differential of the holonomy (without the graded modifications we have worked into our definition), 
\begin{equation*}
 d\Psi = - \int \kappa + (-\int \kappa \omega + \int J\omega \kappa) + \dotsc + \sum_{i+j=r-1}(-1)^{i+1}\int (J \omega)^i \kappa \omega^j + \dotsc + - p_0^*\omega \wedge \Psi + J\Psi \wedge p_1^*\omega,
\end{equation*}
in terms of $\kappa = d\omega - J\omega \wedge \omega$, the \emph{curvature} of $\omega$.

An analog of the above calculation can be proved with two basic lemmas which are modifications of Chen's.  Let $\partial \pi$ be the composition  
\begin{equation*}
 \partial \pi := \pi \circ (in \times id)\colon \,\,\,\, \partial \Delta^k \times X \mto{in \times id} \Delta^k \times X \mto{\pi} X.
\end{equation*} Also define 
\begin{equation*}
 \pi_*(f(t,x)_I \otimes dVol_{\Delta^k} \wedge dx^I) := (\int_{\Delta^k}(-1)^{k \lvert f \rvert}f(t,x)_I dVol_{\Delta^k}) \otimes dx^I
\end{equation*}
---alternating integration along the fiber.  Then we have,

\begin{Lem}
\begin{equation*}
 \pi_* \circ d - (-1)^k d \circ \pi_* = (\partial\pi)_* \circ (in \times id)^*. 
\end{equation*}
And for any $A \in \End^\bullet(V^\bullet) \otimes \A(PM), \,\,\,\, B \in \End^{\bullet}(V^{\bullet})\otimes \A(\Delta^k \times PM)$,
\begin{equation*}
\begin{array}{lcl}
 \pi_*(\pi^*(A)\circ B) &=& (-1)^{k T(A)}A \circ \pi_* B,\\
 \pi_*(B \circ \pi^*A)  &=& (\pi_*B) \circ A.
\end{array}
\end{equation*}
\end{Lem}
\begin{proof} Exercise \end{proof}

\begin{Prop}
\begin{align*}
 d \int \omega_1 \dotsc \omega_r =& \sum_{i=1}^{r} (-1)^{i} \int T\omega_1 \dotsc d\omega_i \omega_{i+1} \dotsc \omega_r\\
 +& \sum_{i=1}^{r-1} (-1)^i \int T\omega_1 \dotsc (T\omega_i \circ \omega_{i+1}) \omega_{i+2} \dotsc \omega_r \\
 +& p_1^*\omega_1 \circ \int \omega_2 \dotsc \omega_r - T(\int \omega_1 \dotsc \omega_{r-1})\circ p_0^*\omega_r.
\end{align*}
\end{Prop}
\begin{proof}
 Using the previous lemma, and the definition of iterated integrals, we have
\begin{multline*}
 d \int \omega_1 \dotsc \omega_r = (-1)^{\spadesuit} d \circ \pi_* \circ \mu \circ \ev_r^* \circ \iota (\omega_1 \otimes \dotsc \otimes \omega_r) = \\
 = (-1)^{\spadesuit+r} \pi_* \circ d \circ \mu \circ \ev_r^* \circ \iota (\omega_1 \otimes \dotsc \otimes \omega_r)\\
 - (-1)^{\spadesuit+r}(\partial \pi)_* \circ (in \times id)^* \circ \mu \circ \ev_r^* \circ \iota (\omega_1 \otimes \dotsc \otimes \omega_r). 
\end{multline*}
Note that the faces of $\Delta^k$ are the sets $\{ t_1 = 1\}$, $\{ t_k = 0\}$, and $\{ t_i = t_{i+1} \}$ and that $d$ commutes with $\mu$ in the (Koszul) graded sense.  Hence this  formula expands/reduces to
\begin{multline*}
(-1)^{\spadesuit+r} \pi_* \circ \mu \circ \ev_r^* \circ \iota (\sum_i T\omega_1 \otimes \dotsc \otimes d\omega_i \otimes \dotsc \otimes \omega_r) \\
 -(-1)^{\spadesuit+r}(\partial \pi)_* \circ (in \times id)^* \circ \mu \circ \ev_r^* \circ \iota (\omega_1 \otimes \dotsc \otimes \omega_r).
\end{multline*}
\begin{multline*}
= \sum_i^r (-1)^i \int T\omega_1\dotsc d\omega_i \dotsc \omega_r \\
 \shoveleft{-(-1)^{\spadesuit+r} [\int_{t_1=1} \mu( p_1^* \omega_1 \otimes \ev_{r-1}^* \circ \iota (\omega_2 \otimes \dotsc \otimes \omega_r))}\\
 \shoveleft{+\int_{t_k=0} \mu (\ev_{r-1}^* \circ \iota (\omega_1 \otimes \dotsc \otimes \omega_{r-1}) \otimes p_0^*\omega_r)}\\
 + \sum_i \int_{t_i = t_{i+1}} \mu (\ev_{r-1}^* \circ \iota (\omega_1 \otimes \dotsc \otimes \omega_i \circ \omega_{i+1} \otimes \dotsc \otimes \omega_r))].
\end{multline*}
\begin{multline*}
 = \sum_i^r (-1)^i \int T\omega_1\dotsc d\omega_i \dotsc \omega_r\\
 \shoveleft{-(-1)^{\spadesuit+r}  [(-1)^{T(\omega_1)(r-1)} p_1^* \omega_1 \circ \int_{t_1=1} \mu( \ev_{r-1}^* \circ \iota (\omega_2 \otimes \dotsc \otimes \omega_r))}\\
 \shoveleft{+ \int_{t_k=0} \mu (\ev_{r-1}^* \circ \iota (\omega_1 \otimes \dotsc \otimes \omega_{r-1})) \circ p_0^*\omega_r}\\  
 + \sum_i \int_{t_i = t_{i+1}} \mu (\ev_{r-1}^* \circ \iota (\omega_1 \otimes \dotsc \otimes \omega_i \circ \omega_{i+1} \otimes \dotsc \otimes \omega_r))].
\end{multline*}
\begin{multline*}
 =\sum_i^r (-1)^i \int T\omega_1\dotsc d\omega_i \dotsc \omega_r +  p_1^* \omega_1 \circ \int \omega_2 \dotsc \omega_r\\
 -(-1)^{\sum_{i=1}^{r-1}(T(\omega_i)-1)}\int \omega_1 \dotsc \omega_{r-1} \circ p_0^*\omega_r\\
- \sum_i (-1)^{\sum_{j=1}^i (T(\omega_j)-1)}\int \omega_1 \dotsc (\omega_i \circ \omega_{i+1})  \dotsc \omega_r.
\end{multline*}
\begin{multline*}
= \sum_i^r (-1)^i \int T\omega_1\dotsc d\omega_i \dotsc \omega_r
 + p_1^* \omega_1 \circ \int \omega_2 \dotsc \omega_r \\
-T(\int \omega_1 \dotsc \omega_{r-1}) \circ p_0^*\omega_r
+ \sum_{i=1}^r\int T\omega_1 \dotsc (T\omega_i \circ \omega_{i+1})  \dotsc \omega_r.
\end{multline*}

On the subset of the path space with fixed endpoints, $PM(x_0, x_1)$, the pullbacks $p_i^*$ kill all but $0$-forms.  Hence we get 
\begin{multline*}
d \int \omega_1 \dotsc \omega_r = \\
\shoveleft{\sum_i^r (-1)^i \int T\omega_1\dotsc d\omega_i \dotsc \omega_r + p_1^* \omega_1^0 \circ \int \omega_2 \dotsc \omega_r}\\
 -T(\int \omega_1 \dotsc \omega_{r-1}) \circ p_0^*\omega_r^0 + \sum_{i=1}^r(-1)^i\int T\omega_1 \dotsc (T\omega_i \circ \omega_{i+1})  \dotsc \omega_r.
\end{multline*}
\end{proof} 

If $\omega$ has total degree $1$, then 
\begin{equation*}
 d \int (\omega)^r = \sum_{i+j+1=r} \int (\omega)^i d\omega (\omega)^j + \sum_{i+j+2=r}\int (\omega)^i (\omega \circ \omega) (\omega)^j 
 + p_1^*\omega^0 \circ \int (\omega)^{r-1} - \int (\omega)^{r-1}\circ p_0^* \omega^0.
\end{equation*}

So for $\omega$ in $\Gamma \End^{1-k}(V)\otimes_{\A^0}\A^k$, with $\Psi$ the holonomy of the local connection $d-\omega$,
\begin{equation}\label{ChenMod}
 d\Psi = [\int \varkappa + (\int \varkappa \omega + \int \omega \varkappa) + \dotsc
 + \sum_{i+j=r-1}\int (\omega)^i \varkappa \omega^j + \dotsc] +  -p_1^*\omega \circ \Psi + \Psi \circ p_0^*\omega,
\end{equation} 
in which $\varkappa := (d-\omega)\circ (d-\omega)= -d\omega - T\omega \circ \omega = -d\omega + \omega \circ \omega$ is the \emph{curvature} of $\omega$.  Note that if $\varkappa =0$ then we have 
\begin{equation}\label{FlatEqn}
 d\Psi = -p_1^*\omega \circ \Psi + \Psi \circ p_0^*\omega.
\end{equation} 
On $PM(x_0,x_1)$ this reduces further to
\begin{equation} \label{keystone}
 d\Psi =-p_0^*A^0\circ \Psi + \Psi \circ p_1^*A^0.
\end{equation} 
The condition $\varkappa = 0$ locally amounts to the series of equations
\begin{equation*}
 \begin{gathered}
A^0\circ A^0=0\\
A^0\circ A^1+A^1\circ A^0 =(d A^0)\\
\dotsc\\
\sum_{i=0}^{q+1}A^i \circ A^{q-i+1} =(dA^{q})\\
\dotsc
 \end{gathered}
\end{equation*}
which is identical to the flatness condition $\E \circ \E=0$ in $\Perf$.  With the help of the Stokes' formula the equation (\ref{keystone}) is equivalent to the integral form
\begin{equation*}
 -A^0_{x_1}\circ \int_{I^q}h^*\Psi_q +(-1)^q(\int_{I^q}h^*\Psi_q)\circ A^0_{x_0} = \int_{\partial I^q} h^*\Psi_{q-1},
\end{equation*}
in which $h\colon I^q \rightarrow P(M,x_0,x_1)$ is any $q$-family of paths inside a trivializing patch.

\subsection{Holonomy With Respect to the Pre-triangulated Structure}\label{sec:3.5}
\subsubsection{Holonomy with Respect to the Shift}\label{sec:3.5.1}
Let $(E^{\bullet},\E)$ be an element of $\Perf$, $d-A$ a local coordinate description, and $\Psi$ its associated holonomy.  The integral of the holonomy over a $k$-cube commutes with the shift functor.  Considering a particular term in the holonomy with form degree $k$,
\begin{multline*}
 \int_{I^k}(\int (A[q])^j) v[q] = (-1)^{(\lvert v \rvert + q)(j+k)}\int_{I^k}(\int (A[q]^j(v[q]))) =\\
=(-1)^{\lvert v \rvert (j+k) + qk}\int_{I^k}(\int A^j(v))= (-1)^{qk}\int_{I^k}(\int A^j)(v)[q].
\end{multline*}
The first sign shows up because the form-degree of the integrand is reduced by $j$ in the integral.  Later we will see $I^k$ as a cube in $PM$ induced in a particular way from a simplex $\sigma_{k+1}$ in $M$.  Likewise, supposing $\phi$ is a degree $p$ morphism between $E$, $F$ locally represented by the matrix-valued forms $A$, $B$ respectively, we see
\begin{multline*}
\int_{I^k}(\int B[q]^i \phi[p+q-1] A[p+q-1])v[p+q-1]= (-1)^{(\lvert v \rvert + p+q-1)(i+j+1 + k)} \cdot \\
\begin{array}{l}
\cdot \int_{I^k}(\int (B[q]^i\phi[p+q-1]A[p+q-1]^j(v[p+q-1])))\\
= (-1)^{(\lvert v \rvert +p-1)(i+j+1+k) +q(k+1)}\int_{I^k}(\int B^i\phi[p-1]A^i[p-1]v[p-1])\\
= (-1)^{q(k+1)}\int_{I^k}(\int B^i\phi[p-1]A^i[p-1])v[p-1]. 
\end{array}
\end{multline*}
 
\subsubsection{Holonomy of a Cone}\label{sec:3.5.2}

Suppose we have a morphism in $\Perf$, i.e. an element $\phi$ of total degree $q$ of
\begin{equation*}
 \Hom_{\Perf}^q(E_1,E_2) = \lbrace \phi\colon E_1 \otimes_{\mathcal{A}^0} \mathcal{A}^{\bullet} \rightarrow E_2 \otimes_{\mathcal{A}^0} \mathcal{A}^{\bullet} \vert \phi(ea) = (-1)^{q\lvert a \rvert}\phi(e)a \rbrace.
\end{equation*}
The differential is defined
\begin{equation*}
 d\phi := E_2 \circ  \phi - (-1)^{\lvert \phi \rvert} \phi \circ E_1.
\end{equation*}
We can construct the cone complex associated to $\phi$, $(C(\phi)^\bullet,D^\phi)$,
\begin{equation*}
 C(\phi)^k = (E_1^{k+1-q}\oplus E_2^k), \,\,\,\, D^\phi = \begin{pmatrix}
 (-1)^{1-q}\mathbb{E}_1  & 0 \\
 \phi & \mathbb{E}_2
\end{pmatrix}.
\end{equation*}
Note that $D^\phi$ is flat iff $\phi$ is a closed morphism.  

In a trivializing coordinate patch write $D^{\phi} = d- \omega$ and denote the corresponding $\Z$-connection holonomy by $\Psi^{\phi}$.  Then applying our Chen-formula (eqn. \ref{ChenMod}) for $d\Psi$ on $PM(x_0,x_1)$, we calculate,
\begin{equation*}
 d\Psi^{\phi} = \biggl[\int \varkappa + (\int \varkappa \omega + \int \omega \varkappa) + \dotsc
 + \sum_{i+j=r-1}\int \omega^i \varkappa \omega^j + \dotsc \biggr] -p_1^*\omega^0 \circ \Psi^{\phi} + \Psi^{\phi}\circ p_0^*\omega^0.
\end{equation*}
And since $D^{\phi}_{11}$, and $D^{\phi}_{22}$ are flat, it is evident that 
\begin{equation*}
\varkappa = -d\omega + \omega \circ \omega = 
\left(\begin{smallmatrix} 0 & 0\\ d\phi & 0
\end{smallmatrix} \right).
\end{equation*}
Then the 21-component is:
\begin{equation}
d\Psi^{\phi}_{21} = \biggl [\int d\phi + \dotsc + \sum_{i+j+2=r} \int (B)^i d\phi ((-1)^{1-q}A)^j + \dotsc \biggr] -p_1^*B^0\circ \Psi^{\phi}_{21} +(-1)^{1-q}\Psi^{\phi}_{21}\circ p_0^*A^0.
\end{equation}

Alternately, we first take $d\phi$ in $\Perf$ and take the holonomy of its cone $C(d\phi)$.  We already showed (eqn. \ref{FlatEqn}) that since $D^{d\phi}$ is flat, 
\begin{equation*}
d\Psi^{d\phi} = -p_1^*\omega^{d\phi,0}\circ \Psi^{d\phi} + \Psi^{d\phi}\circ p_0^*\omega^{d\phi,0}.
\end{equation*}
Expanding the definition we have,
\begin{multline*}
\Psi^{d\phi}_{21} = \int d\phi + \dotsc \sum_{i+j+1 = r} \int (B)^i d\phi ((-1)^qA)^j + \dotsc = \\
 = \int d\phi + \dotsc \sum_{i+j+1 = r} (-1)^j\int (B)^i d\phi ((-1)^{1-q}A)^j + \dotsc
\end{multline*}
According to the signs in our definition of iterated integrals, the sign changes by $j+1$ if we switch $d\phi$ from degree $1$ to degree $2$.  Thus, the above formula becomes
\begin{equation*}
 \Psi^{d\phi}_{21}= -\int d\phi - \dotsc -\sum_{i+j+1 = r}\int (B)^i d\phi ((-1)^{1-q}A)^j+ \dotsc
\end{equation*} 
And consequently,
\begin{equation}
 d\Psi^{\phi}_{21} = -\Psi^{d\phi}_{21} - p_1^*B^0\circ \Psi^{\phi}_{21}+ (-1)^{\lvert \phi \rvert-1}\Psi^{\phi}_{21}\circ p_0^* A^0.
\end{equation}

Going further, we can consider generalized cones associated to any string of morphisms 
\begin{equation*}
 \phi_n \otimes \dotsc \otimes \phi_1  \in \Perf(E_{n-1},E_{n})\otimes \dotsc \otimes \Perf(E_{0},E_1).                                                                    
\end{equation*}
Let $p_i$ denote the total degree of $\phi_i$, and define $D^{\phi_n \otimes \dotsc \otimes \phi_1}$ to be the total degree $1$ endomorphism of  
\begin{equation*}
  E_0^\bullet[n-\sum_1^n p_i] \oplus E_1^\bullet[n-1-\sum_2^{n}p_i] \oplus \dotsc \oplus E_n^\bullet
\end{equation*}
given by
\begin{equation*}
 D^{\phi_n \otimes \dotsc \otimes \phi_1} = \left(\begin{smallmatrix}  \E_0[n-\sum_{1}^np_i] & 0 & 0 &\dotsc & \dotsc & 0 & 0 \\
								      \phi_1[n-\sum_{1}^n p_i] & \E_1[n-1-\sum_{2}^np_i] & 0 & \dotsc & \dotsc & 0 &0\\
								      0 & \phi_2[n-1-\sum_2^{n} p_i] & \E_2[n-2-\sum_3^{n} p_i] & \dotsc & \dotsc & 0&0\\
								      \dotsc & \dotsc & \dotsc & \dotsc & \dotsc & \dotsc & \dotsc \\
								      0 & 0 & \dotsc & \dotsc & \phi_{n-1}[2-p_{n-1}-p_n] & \E_{n-1}[1-p_n] & 0\\
								      0 & 0 & \dotsc & \dotsc & 0 & \phi_n[1-p_n]& \E_n 
\end{smallmatrix} \right).
\end{equation*}
\begin{Defn}\label{GenHomCone}
We call 
\begin{equation*}
C(\phi_n \otimes \dotsc \otimes \phi_1) = ( E_0^\bullet[n-\sum_1^n p_i] \oplus E_1^\bullet[n-1-\sum_2^{n}p_i] \oplus \dotsc \oplus E_n^\bullet,D^{\phi_n \otimes \dotsc \otimes \phi_1})
\end{equation*}
the generalized homological cone associated to this $n$-tuple of morphisms.
\end{Defn}
In a local trivialization $ D^{\phi_n \otimes \dotsc \otimes \phi_1}= d-\omega, with$
\begin{equation*}
\omega = 
\left(\begin{smallmatrix}  A_0[n-\sum_{1}^np_i] & 0 & 0 &\dotsc & \dotsc & 0 & 0 \\
								      \phi_1[n-\sum_{1}^n p_i] & A_1[n-1-\sum_{2}^np_i] & 0 & \dotsc & \dotsc & 0 &0\\
								      0 & \phi_2[n-1-\sum_2^{n} p_i] & A_2[n-2-\sum_3^{n} p_i] & \dotsc & \dotsc & 0&0\\
								      \dotsc & \dotsc & \dotsc & \dotsc & \dotsc & \dotsc & \dotsc \\
								      0 & 0 & \dotsc & \dotsc & \phi_{n-1}[2-p_{n-1}-p_n] & A_{n-1}[1-p_n] & 0\\
								      0 & 0 & \dotsc & \dotsc & 0 & \phi_n[1-p_n]& A_n 
\end{smallmatrix} \right).
\end{equation*}
Therefore the curvature is
\begin{equation*}
 \varkappa =
\left(\begin{smallmatrix}  0  & 0 & \dotsc &0 & 0 & 0 \\
								      (-1)^{n-1-\sum_2^{n}p_i} d\phi_1[n-\sum_1^np_i] & 0 & \dotsc & 0 & 0 &0\\
								      \phi_2 \circ \phi_1[n-\sum_1^n p_i] & \dotsc & \dotsc & 0&0&0\\
								      \dotsc  & \dotsc & \dotsc & \dotsc & \dotsc & \dotsc \\
								      0  & 0 & \dotsc & (-1)^{1-p_n}d\phi_{n-1}[2-p_{n-1}-p_n] & 0 & 0\\
								      0 & 0 & \dotsc & \phi_n \circ \phi_{n-1}[2-p_{n-1}-p_n] & d\phi_n[1-p_n]& 0
\end{smallmatrix} \right).
\end{equation*}
We call the holonomy of this connection $\Psi^{\phi_n \otimes \dotsc \otimes \phi_1}$.  Using the modified Chen formula (eqn. \ref{ChenMod}), we consider the $n+1,1$-component of the differential of this form:
\begin{equation}
 d\Psi^{\phi_n \otimes \dotsc \otimes \phi_1}_{n+1,1}\\ 
 = [\sum_{i,j}\int\omega^i \varkappa \omega^j]_{n+1,1} -p_1^* \omega_{n+1,1}^0 \circ \Psi^{\phi_n \otimes \dotsc \otimes \phi_1}_{n+1,1} + \Psi^{\phi_n \otimes \dotsc \otimes \phi_1}_{n+1,1} \circ p_0^*\omega_{n+1,1}^0.
\end{equation}
Considering the terms in the first part of the sum, (the shifting is suppressed for clarity)
\begin{multline*}
 \sum_{i,j}[\int \omega^i \varkappa \omega^j]_{n+1,1} =  \sum_{k=1}^{n-1}\biggl[\\
\sum_{(i_0,\dotsc,\hat{i_{k}},\dotsc,i_n)} \int A_n^{i_n} \phi_n A_{n-1}^{i_{n-1}} \phi_{n-1} \dotsc \phi_{k+2} A_{k+1}^{i_{k+1}} (\phi_{k+1} \circ \phi_k) A_{k-1}^{i_{k-1}} \phi_{k-1} \dotsc \phi_1 A_0^{i_0}\biggr] \\
 + \sum_{k=1}^{n} \sum_{(i_0,\dotsc,i_n)} \int A_{n}^{i_n} \phi_n A_{n-1}^{i_{n-1}} \phi_{n-1} \dotsc \phi_{k+1} A_{k}^{i_{k}} d\phi_k A_{k-1}^{i_{k-1}} \phi_{k-1} \dotsc \phi_1 A_0^{i_0}.
\end{multline*}
Now, we can recognize inside this series the terms of
\begin{equation*}
 \Psi^{\phi_n \otimes \dotsc \otimes d\phi_k \otimes \dotsc \otimes \phi_1}_{n+1,1}\,\, \text{and} \,\, \Psi^{\phi_n \otimes \dotsc \otimes \phi_{k+1} \circ \phi_{k} \otimes \dotsc \otimes \phi_1}_{n+1,1}.
\end{equation*}
These are the holonomies of the cone $D^{\phi_n \otimes \dotsc \otimes d\phi_k \otimes \dotsc \otimes \phi_1}= d-\omega$ with
\begin{equation*}
\omega =
 \left(\begin{smallmatrix}  A_0[n-\sum_{1}^np_i+1] & 0 &\dotsc &\dotsc & 0 & 0 & 0 \\
								      \phi_1[n-\sum_{1}^n p_i+1] & \dotsc & \dotsc & \dotsc & 0 & 0 &0\\
								      \dotsc & \dotsc & \dotsc & \dotsc & \dotsc & \dotsc & \dotsc \\
								      0 & \dotsc & d\phi_k[n-k+1-\sum_k^{n} p_i+1] & A_k[n-k+1-\sum_{k+1}^{n} p_i] & \dotsc  & 0 & 0 \\
								      \dotsc & \dotsc & \dotsc & \dotsc & \dotsc & \dotsc & \dotsc \\
								      0 & 0 & \dotsc & \dotsc & \dotsc & A_{n-1}[1-p_n] & 0\\
								      0 & 0 & \dotsc & \dotsc & 0 & \phi_n[1-p_n]& A_n 
\end{smallmatrix} \right).
\end{equation*}
As before, shifting the terms to the right of $d\phi_k$ has the effect of changing the sign on all of the $A$'s to the left of $d\phi_k$ and leaving the sign on the $\phi's$ unchanged.  But in doing so we change the degree of $d\phi_k$ by one, thus introducing a sign change for every term to the right of $d\phi_k$ and one more from the alternation of the integral.  So the total change is $(-1)^{n-k+1}$.  But the sign in the definition of the iterated integral again accounts for this $n-k$.  All together we have computed,
\begin{multline}\label{TwoOneComp}
d\Psi^{\phi_n \otimes \dotsc \otimes \phi_1}_{n+1,1} = -\sum_{k=1}^{n-1}(-1)^{n-k-1-\lvert \phi_n \otimes \dotsc \otimes \phi_{k+2} \rvert}\Psi^{\phi_n \otimes \dotsc \otimes \phi_{k+1} \circ \phi_{k} \otimes \dotsc \otimes \phi_1}_{n+1,1}\\
 -\sum_{k=1}^{n} (-1)^{n-k- \lvert \phi_n \otimes \dotsc \otimes \phi_{k+1}\rvert}\Psi^{\phi_n \otimes \dotsc \otimes d\phi_k \otimes \dotsc \otimes \phi_1}_{n+1,1}\\
  -p_1^* A_n^0 \circ \Psi^{\phi_n \otimes \dotsc \otimes \phi_1}_{n+1,1} + (-1)^{\lvert \phi_n \otimes \dotsc \otimes \phi_n \rvert}\Psi^{\phi_n \otimes \dotsc \otimes \phi_1}_{n+1,1} \circ p_0^*A_0^0.
\end{multline}

\subsection{Cubes to Simplices}\label{sec:3.7}

Now we want to integrate over simplices rather than cubes, which will involve realizing any simplex as a family of paths with fixed endpoints.  This construction is essentially due to Adams \cite{MR0090045}.  It was modified for use in the differentiable setting by Chen, e.g. \cite{MR0454968}, and is described in a detailed manner by Igusa in (arXiv:0912.0249v1).  We only outline it here, citing the relevant properties.  Throughout the section $P$ is the path space functor.

Given a geometric $k$-simplex, $\sigma \colon \Delta^k \rightarrow M$, we want to realize this as a factor of a $(k-1)$-family of paths into $M$.  That is, we produce a map $\theta \colon I^k \rightarrow \Delta^k$ which then can be viewed as a family of paths $\theta_{(k-1)}\colon I^{k-1} \rightarrow P\Delta^k$.  This map is factored into two parts,
\begin{equation*} 
I^k \xrightarrow{\lambda} I^k \xrightarrow{\pi_k} \Delta^k.
\end{equation*}  
Here $\pi_k$ is an order-preserving retraction.  $\lambda$ is given by the map $\lambda_w \colon I \rightarrow I^k$ parametrized by $w \in I^{k-1}$.  The result, is an $I^{k-1}$-family of paths in $I^k$ (we call this $\lambda_{(k-1)}\colon I^{k-1} \rightarrow PI^{k}$) each starting at $(w_1,w_2,\dotsc,w_{k-1},1)$ and ending at $(0,0,\dotsc,0)$.  When post-composed with $\pi_k$ we get a $(k-1)$-family of paths in $\Delta^k$ which start at $\sigma_k$ and end at $\sigma_0$.  Define $\theta_{(k-1)}\colon I^{k-1} \rightarrow P\Delta^k$ by $P\pi_k \circ \lambda_{(k-1)}.$ \\
We restate the characteristic properties of such a factorization c/o Igusa:

\begin{tabular}{cl}
$\bullet$ & If $x\leq X^{\prime}$ in the sense that $x_i \leq x^{\prime}_i$ for all i, then $\pi_k(x) \leq \pi_k(x^{\prime})$.\\  		          & Furthermore, $\pi_k(x) \geq x$. \\
$\bullet$ & $\pi_k$ sends $\partial_i^+I^k = \{x \in I^k \lvert x_i = 1\}$ to the back $k-i$ face of $\Delta^k$.\\ 
          & This face is spanned by $\{ v_i, \dotsc, v_k \}$ and given by the equation $y \geq v_i$.\\
$\bullet$ & $\pi_k$ sends $\partial_i^-I^k = \{ x \in I^k \lvert x_i = 0\}$ onto $\partial_i \Delta^k = \{ y\in \Delta^k \lvert y_i = y_{i+1}\}$. \\
\end{tabular}\newline
and,

\begin{tabular}{cl}
 $\bullet$ & The adjoint of $\theta_{(k)}$ is a piecewise-linear epimorphism $I^k \twoheadrightarrow \Delta^k.$\\
 $\bullet$ & For each $w \in I^{k-1}$, $\theta_w$ is a path from $\theta_w(0) = v_k$ to $\theta_w(1) = v_0$.\\
 $\bullet$ & $\theta_w$ passes through the vertex $v_i$ iff $w_i = 1$.\\
 $\bullet$ & $\theta_{(k)}$ takes each of the $2^{k-1}$ vertices of $I^{k-1}$ to the shortest path\\
           & from $v_k$ to $v_0$ passing through the corresponding subset\\
	   & $\{ v_1 \dotsc v_{k-1} \}$.\\
\end{tabular}

\section{An $A_{\infty}$-quasi-equivalence}\label{sec:4}
In this section we establish our Riemann--Hilbert correspondence for infinity-local systems.  Recall that $\Cal$ is the dg-category of cochain complexes over $\R.$
\begin{Thm}\label{maintheorem}
There is an $A_{\infty}$-functor,
\[
\mathcal{RH}\colon\Perf\to \mathsf{Loc}^\Cal_\infty(\Pinf),
\]
which is a quasi-equivalence.
\end{Thm}

Recall that a dg-category is a special case of an $A_\infty$-category. An $A_\infty$-functor $F$ (with components $\{F_i\}$ indexed by valence) between two dg-categories $A$ and $B$ will satisfy the $A_\infty$-condition,
\begin{equation*}
 \sum_{a+b=k}\mu_B \circ F_a \otimes F_b +  F_k \circ d_A = d_B \circ F_k + F_{k-1} \circ \sum_{i+j+2=k}[I^{\otimes i} \otimes \mu_A \otimes I^{\otimes j}], 
\end{equation*}
in which $\mu$ represents the multiplications and $d$ the respective differentials in these dg-categories.  This is written succinctly on account of the abundance of surveys of $A_\infty$-structures in the literature ---see \cite{MR2441780} for more details.

\subsection{The Functor $\RH$}\label{sec:4.1}

On objects the functor 
\begin{equation*}
\RH_0\colon \Ob(\Perf) \rightarrow \Ob(\mathsf{Loc}_{\infty}^{\Cal}(\Pinf))
\end{equation*}
is described as follows.  Given an element $(E^{\bullet},\E) \in \Perf$ take the corresponding graded bundle $V$ over $M$ with a $\Z$-graded connection $\E$.  Define an infinity-local system by the assignments,
\begin{align*}
\RH_0((E^{\bullet},\E))(x) &:= (V_x,\E^0_x),\\
\RH_0((E^{\bullet},\E))(\sigma_k) &:= \int_{I^{k-1}} (-1)^{(k-1)(K\Psi)}\theta^*_{(k-1)}(P\sigma)^*\Psi.
\end{align*}
That is, assign to each $k$-simplex the integral of the higher holonomy integrated over that simplex (understood as a $(k-1)$-family of paths).  The result is a degree-$(1-k)$ homomorphism from the fiber over the endpoint to the fiber over the starting point of the simplex.  To a $0$-simplex this yields a degree $1$ map in the fiber over that point which we shall see will be a differential as a result of the flatness of the $\Z$-graded connection.  To a $1$-simplex (a path) we get the usual parallel transport of the underlying graded connection.  Flatness will imply that this is a cochain map with respect to the differentials on the fibers over the endpoints of the path.
 
So far we only have a simplicial set map from $\Pinf$ to a simplicial set consisting of simplices in $\Cal$ which are not necessarily homotopy coherent.  Call this map $F$.  Since we are integrating a flat $\Z$-graded connection, the holonomy $\Psi$ satisfies,
\begin{equation*}
d\Psi = -p_0^*A^0 \circ \Psi + \Psi \circ p_1^*A^0.
\end{equation*}
Via Stokes' Theorem, $F$ satisfies the local system condition,
\begin{equation}
\begin{gathered}
 \E^0 \circ F_k(\sigma) -(-1)^kF_k(\sigma)\circ \E^0 = \sum_{i=1}^{k-1}(-1)^{i}F_{k-1}(\sigma_0,\dotsc,\hat{\sigma_i},\dotsc,\sigma_k)+\\
-\sum_{i=1}^{k-1}(-1)^{i}F_{i}(\sigma_0,\dotsc,\sigma_i)\circ F_{k-i}(\sigma_i,\dotsc,\sigma_k),\\\\
\end{gathered}
\end{equation}
which is the same as the required Maurer--Cartan/twisting-cochain condition,
\begin{equation*}
 dF + \hat{\delta} F + F \cup F = 0.
\end{equation*}
Proving this relation amounts to the task of figuring out what $\int_{\partial I^{q-1}}h^*\Psi$ is in the case that h is the map constructed above which factors through $\sigma$.  That is we must relate $\partial I^{q-1}$ to $\partial \Delta^{q-1}$.  Igusa works this out elegantly in his preprint and obtains (If we write $I(\sigma_k):= \int (-1)^{(k-1)K\Psi} \theta^*(P[\sigma_k])^*(\Psi)$),
\begin{equation*}
\int (-1)^{(k-1)K(d\Psi)}\theta^*(P[\sigma_k])^*(d\Psi) = -\hat{\delta} I - I \cup I.
\end{equation*}

Now we can describe the map $\RH_n$,
\begin{equation*}
 \RH_n \colon \Perf(E_{n-1},E_{n})\otimes \dotsc \otimes \Perf((E_{0},E_1))
 \rightarrow \Loc(X)(\RH_0(E_0),\RH_0(E_n))[1-n],
\end{equation*}
on composable $n$-tuples of morphisms.  Given a tuple $\phi_n \otimes \dotsc \otimes \phi_1$, assign to it the generalized homological cone (Defn. \ref{GenHomCone}), and its associated holonomy $\Psi^{\phi_n \otimes \dotsc \otimes \phi_1}$.  Then define
\begin{equation*}
 \RH_n(\phi_n \otimes \dotsc \otimes \phi_1)(\sigma_k) := \RH_0(C(\phi_n \otimes \dotsc \otimes \phi_1)(\sigma_k))_{n+1,1}.
\end{equation*}
Note that applying $\RH_0$ to the cone $C(\phi_n \otimes \dotsc \otimes \phi_1)$ does not necessarily yield an infinity local system.  $\RH_0$ is perfectly well-defined as a holonomy map on any $\Z$-connection regardless of flatness.  Flatness implies that the image is an infinity-local system.

\begin{Thm}
 The maps $\{\RH_i\}$ define an $A_\infty$-functor
\begin{equation*}
\RH\colon \Perf \rightarrow  \mathsf{Loc}_{\infty}^{\Cal}(\pi_{\infty}).
\end{equation*} 
\end{Thm}
\begin{proof}
Given a tuple of morphisms $\phi := \phi_n \otimes \dotsc \otimes \phi_1 \in \Perf(E_{n-1},E_{n})\otimes \dotsc \otimes \Perf(E_{0},E_1)$, denote the holonomy associated to the generalized homological cone $C(\phi)$ by $\Psi^{\phi_n \otimes \dotsc \otimes \phi_1}$.  Locally write $D^\phi = d-\omega$.

We already calculated that (on $PM(x_0,x_1)$), (eqn. \ref{TwoOneComp})
\begin{multline*}
 -d\Psi^{\phi_n \otimes \dotsc \otimes \phi_1}_{n+1,1} -p_0^* \omega_{n+1,1}^0 \circ \Psi^{\phi_n \otimes \dotsc \otimes \phi_1}_{n+1,1} + \Psi^{\phi_n \otimes \dotsc \otimes \phi_1}_{n+1,1} \circ p_1^*\omega_{n+1,1}^0= \\
  \shoveleft{=\sum_{k=1}^{n-1}(-1)^{n-k-1-\lvert \phi_n \otimes \dotsc \otimes \phi_{k+2} \rvert}\Psi^{\phi_n \otimes \dotsc \otimes \phi_{k+1} \circ \phi_{k} \otimes \dotsc \otimes \phi_1}_{n+1,1}} + \sum_{k=1}^{n} (-1)^{n-k-\lvert \phi_n \otimes \dotsc \otimes \phi_{k+1} \rvert}\Psi^{\phi_n \otimes \dotsc \otimes d\phi_k \otimes \dotsc \otimes \phi_1}_{n+1,1}.
\end{multline*}
Thus, applying $\int (-1)^{K(\Psi)J(\bullet)}\theta^*(P[\bullet])^*(\Psi)$ to both sides yields,
\begin{multline*}
\biggl [\RH_0(C(\phi)) \cup \RH_0(C(\phi)) +\hat{\delta}\RH_0(C(\phi)) + d\RH_0(C(\phi))\biggr ]_{n+1,1} =\\
 \shoveleft{= \sum_{k=1}^n(-1)^{n-k-\lvert \phi_n \otimes \dotsc \otimes \phi_{k+1} \rvert} \RH_{n-1}(\phi_n \otimes \dotsc \otimes d\phi_k \otimes \dotsc \phi_1)}\\
 + \sum_{k=1}^{n-1}(-1)^{n-k-1-\lvert \phi_n \otimes \dotsc \otimes \phi_{k+2} \rvert} \RH_{n-1}(\phi_n \otimes \dotsc \otimes (\phi_{k+1} \circ \phi_{k}) \otimes \dotsc \otimes \phi_1).
\end{multline*}
Observe (denoting $\phi_k \otimes \dotsc \otimes \phi_l$ by $\phi_{k,l}$),
\begin{multline*}
\biggl [\RH_0(C(\phi))\cup\RH_0(C(\phi))\biggr ]_{n+1,1}=\\
\shoveleft{=\sum_{i+j=n}\RH_0(C(\phi_{n,i+1}))_{j+1,1} \cup \RH_0(C(\phi_{i,1}))_{i+1,1}[(j-\sum_{i+1}^np_k)]}\\
\shoveleft{+ \RH_0(E_n) \cup \RH_0(C(\phi))_{n+1,1} + \RH_0(C(\phi))_{n+1,1}\cup \RH_0(E_0[n-\lvert \phi \rvert])}\\
\\ 
\shoveleft{=\sum_{i+j=n}(-1)^{(j-\sum_{i+1}^np_k)}\RH_j(\phi_{n,i+1}) \cup \RH_i(\phi_{i,1})}\\
+ \RH_0(E_n) \cup \RH_0(C(\phi))_{n+1,1}\\
+ (-1)^{n-\lvert \phi \rvert}\RH_0(C(\phi))_{n+1,1} \cup \RH_0(E_0).
\end{multline*}
By definition,
\begin{multline*}
 D_{\Loc(X)}(\RH_n(\phi)) =\\
 \RH_0(E_n) \cup \RH_n(\phi) + (-1)^{n- \lvert \phi \rvert}\RH_n(\phi) \cup \RH_0(E_0) +\hat{\delta}\RH_n(\phi) + d\RH_n(\phi).
\end{multline*}
Hence we get, 
\begin{multline*}
\sum_{i+j=n}(-1)^{j-\rvert \phi_n \otimes \dotsc \otimes \phi_{i+1} \rvert}\RH_j(\phi_n \otimes \dotsc \otimes \phi_{i+1}) \cup \RH_i(\phi_{i}\otimes \dotsc \otimes \phi_1) + D(\RH_n(\phi_n \otimes \dotsc \otimes \phi_1))\\
 \shoveleft{= \sum_{k=1}^n (-1)^{n-k-\lvert \phi_n \otimes \dotsc \otimes \phi_{k+1}\rvert }\RH_{n}(\phi_n \otimes \dotsc \otimes d\phi_k \otimes \dotsc \otimes \phi_1)}\\
 +\sum_{k=1}^{n-1} (-1)^{n-k-1 - \lvert \phi_n \otimes \dotsc \otimes \phi_{k+2} \rvert}\RH_{n-1}(\phi_n \otimes \dotsc \otimes \phi_{k+1} \circ \phi_k \otimes \dotsc \otimes \phi_1).
\end{multline*}
These are the $A_\infty$-relations for an $A_\infty$-functor between two dg-categories understood as $A_\infty$-categories.
\end{proof}
\begin{Prop}
The functor $\RH$ is $A_\infty$-quasi-fully faithful. 
\end{Prop}
\begin{proof} Consider two objects  $E_i=(E_i^\bullet,\mathbb{E}_i)\in \Perf$, $i=1,2$. The chain map, 
\[
\RH_1\colon\Perf(E_1,E_2)\to \mathsf{Loc}^\Cal_\infty(\Pinf)(\RH_0(E_1),\RH_0(E_2)),
\]
induces a map on spectral sequences (prop. \ref{SS}) and \cite{MR2648899}, Theorem 2.5.1. At the $E_1$-level on the $\Perf$ side, we have that $H^*((E_i,\mathbb{E}_i^0))$ are both vector bundles with flat connection, while according to corollary \ref{perfect}, we have $H^*((\RH(E_i),\E^0_i))$ are local systems on $M$. At the $E_2$-term the map is
\begin{equation*}
H^*(M; \Hom(H^*(E_1,\mathbb{E}_1^0),H^*(E_2,\mathbb{E}_2^0)))\to H^*(M;H^*((\RH(E_1),\E^0_1)),H^*((\RH(E_2),\E^0_2))),
\end{equation*}
which is an isomorphism by the ordinary De Rham theorem for local systems. 
\end{proof}

\subsection{$\RH$ is $A_\infty$-essentially surjective}\label{sec:4.2}
We must prove that for any $(F,f)\in \mathsf{Loc}_{\infty}^{\Cal}(\Pinf) $, that there is an 
object $E=(E^\bullet,\mathbb{E})\in \Perf$ such that $\RH_0(E)$ is quasi-isomorphic to $(F,f)$. 
We first define a complex of sheaves on $M$. Let $\underline{\mathbb{R}}$ denote the constant local system, 
and thus an infinity-local system. We also view $\underline{\mathbb{R}}$  as a sheaf of rings with which 
$(M,\underline{\mathbb{R}})$  becomes a ringed space. For an open subset $U\subset M$, let
 $(C_F(U),D)=(\mathsf{Loc}_{\infty}^{\Cal}(\pi_\infty U)(\underline{\mathbb{R}}|_U,F|_U),D)$. Let $(\underline{C}_F,D)$ 
 denote the associated complex of sheaves. Then  $\underline{C}_F$ is soft; see the proof of Theorem 3.15,  \cite{MR0515872}.  
 By corollary \ref{perfect}, $\underline{C}_F$ is a perfect complex of sheaves over $\underline{\mathbb{R}}$.  
 Let $\underline{\A}_M$ denote the sheaf of $C^\infty$ functions and $(\underline{\A}^\bullet, d)$ denote 
 the dg sheaf of $C^\infty$ forms on $M$. Set $\underline{C}_F^\infty =
\underline{C}_F\otimes_{\underline{\mathbb{R}}}\underline{\A}_M$. By the flatness of $\underline{\A}_M$ over
 $\underline{\mathbb{R}}$, $\underline{C}_F^\infty$ is perfect as a sheaf of $\underline{\A}_M$-modules. Now the map
 \[(\underline{C}_F^\bullet,D)\to (\underline{C}_F^\infty\otimes_{\underline{\A}_M}\underline{\A}_M^{\bullet},D\otimes 1+1\otimes d),\]
 is a quasi-isomorphism of sheaves of $\underline{\mathbb{R}}$-modules by the flatness of $\underline{\A}_M$ over 
 $\underline{\mathbb{R}}$. 
 
 \noindent We need the following proposition. 
 \begin{Prop}\label{SoftStuff}
Suppose $(X,\underline{\mathcal{S}}_X)$ is a ringed space, where $X$ is compact and $\underline{\mathcal{S}}_X$ is a soft sheaf of rings. Then 
\begin{enumerate}
\item The global sections functor 
\[
\Gamma\colon\text{Mod-}\underline{\mathcal{S}}_X\to \text{Mod-}\underline{\mathcal{S}}_X(X)
\]
is exact and
establishes an equivalence of categories between the category of sheaves of right $\underline{\mathcal{S}}_X$-modules and the category of right modules over the global sections $\underline{\mathcal{S}}_X(X)$.
\item If $\underline{M}\in\text{Mod-}\underline{\mathcal{S}}_X$ locally has finite resolutions by finitely generated free $\underline{\mathcal{S}}_X$-modules, then $\Gamma(X;\underline{M})$ has a finite resolution by finitely generated projectives. 
\item The derived category of perfect complexes of sheaves $D_{\mbox{perf}}(\text{Mod-}\underline{\mathcal{S}}_X)$ is equivalent 
the derived category of perfect complexes of modules \newline $D_{\mbox{perf}}(\text{Mod-}\underline{\mathcal{S}}_X(X))$.
\end{enumerate}
\end{Prop}
\begin{proof} See Proposition 2.3.2, Expos\'{e} II, SGA6, \cite{MR0354655}. \end{proof}

 \begin{Thm}
 The functor, 
 \[
 \RH\colon\Perf\to \mathsf{Loc}_\infty^\Cal(\Pinf),
 \]
 is $A_\infty$-essentially surjective.
 \end{Thm}
 \begin{proof}  By the proposition, there is a (strictly) perfect complex $(E^\bullet, \mathbb{E}^0)$
 of $\A$-modules along with a quasi-isomorphism
 \begin{equation*}
 e^0\colon(E^\bullet,\mathbb{E}^0)\to (X^\bullet, \mathbb{X}^0) :=(\Gamma(M,\underline{C}_F^\infty), D).  
 \end{equation*}
 Following the argument of Theorem 3.2.7 of \cite{MR2648899}, which in turn is based on arguments from \cite{MR506654}, we construct the higher components $\mathbb{E}^i$ of a $\mathbb{Z}$-graded connection along with the higher components of a morphism $e^i$ at the same time.   
  
We have a $\mathbb{Z}$-graded connection on $X^\bullet$ by
\[
\mathbb{X}:=D\otimes 1+1\otimes d\colon X^\bullet \to X^\bullet\otimes_\A\A^\bullet.
\]
Then we have an induced connection 
\[
\mathbb{H}\colon H^k(X^\bullet,\mathbb{X}^0)\to H^k(X^\bullet,\mathbb{X}^0)\otimes_\A\A^1
\]
for each $k$. 
We use the quasi-isomorphism $e^0$ to transport this connection to a connection, also denoted by $\mathbb{H}$ on $H^k(E^\bullet;\mathbb{E}^0),$
\begin{equation*}
\begin{array}{ccc}
 H^k(E^\bullet;\mathbb{E}^0)& \stackrel{\mathbb{H}}{\to} & H^k(E^\bullet,\mathbb{E}^0)\otimes_\A\A^1 \\
\downarrow e^0 & & \downarrow e^0\otimes 1 \\
H^k(X^\bullet,\mathbb{X}^0)& \stackrel{\mathbb{H}}\longrightarrow &  H^k(X^\bullet,\mathbb{X}^0)\otimes_\A\A^1
\end{array}\end{equation*}
The right vertical arrow above $e^0\otimes 1$ is a quasi-isomorphism because $\A^\bullet$ is flat over $\A$. 
The first step is handled by the following lemma. 
\begin{Lem}
Given a bounded complex of f.g. projective $\A$ modules $(E^\bullet,\mathbb{E}^0)$ with connections
\begin{equation*}
\mathbb{H}\colon H^k(E^\bullet;\mathbb{E}^0)\to H^k(E^\bullet,\mathbb{E}^0)\otimes_\A\A^1,
\end{equation*}
for each $k$, there exist connections
\[
\widetilde{\mathbb{H}}\colon E^k\to E^k\otimes_\A\A^1
\]
lifting $\mathbb{H}$. That is,  
\[
\widetilde{\mathbb{H}}\mathbb{E}^0=(\mathbb{E}^0\otimes 1)\widetilde{\mathbb{H}},
\]
and the connection induced on the cohomology is $\mathbb{H}$.
\end{Lem}
\begin{proof}(of lemma) 
Since $E^\bullet$ is a bounded complex of $\A$-modules it lives in some bounded range of degrees $k\in [N,M]$. Pick an arbitrary connection on $E^M$, $\nabla$. Consider the diagram with exact rows
\begin{equation*}\label{CD1.5}
\begin{array}{lclc}
  E^{M} &\stackrel{j}{ \to} & H^M(E^\bullet,\mathbb{E}^0)&\to 0 \\
    \nabla\downarrow & \stackrel{\theta}{\searrow} &  \mathbb{H}\downarrow &\\
       & & & \\
E^{M}\otimes_\A\A^1 &\stackrel{j\otimes 1}{ \to} & H^M(E^\bullet,\mathbb{E}^0)\otimes_\A\A^1  &\to 0
\end{array}
\end{equation*}
In the diagram, $\theta=\mathbb{H}\circ j-(j\otimes 1)\circ \nabla$ is easily checked to be $\A$-linear and  $j\otimes 1$ is surjective by the right exactness of tensor product. By the projectivity of $E^M$, $\theta$ lifts to 
\[
\widetilde{\theta}\colon E^M\to E^M\otimes_\A\A^1,
\]
so that $(j\otimes 1)\widetilde{\theta}-\theta$. 
Set $\widetilde{\mathbb{H}}=\nabla+\widetilde{\theta}$. With $\widetilde{\mathbb{H}}$ in place of $\nabla$, the diagram above commutes.

Now choose on $E^{M-1}$ any connection $\nabla_{M-1}$. But $\nabla_{M-1}$ does not necessarily satisfy $\mathbb{E}^0\nabla_{M-1}=\widetilde{\mathbb{H}}\mathbb{E}^0=0$. We correct it as follows.  Set $\mu=\widetilde{\mathbb{H}}\mathbb{E}^0-(\mathbb{E}^0\otimes 1)\nabla_{M-1}$. Then $\mu$ is $\A$-linear. Furthermore, 
$\I \mu\subset \I \mathbb{E}^0\otimes 1$; this is because $\widetilde{\mathbb{H}}\mathbb{E}\in \I \mathbb{E}\otimes 1$ since $\widetilde{\mathbb{H}}$ lifts $\mathbb{H}$. So by  projectivity it lifts
to $\widetilde{\theta}\colon E^{M-1}\to E^{M-1}\otimes_\A\A^1$ such that $(\mathbb{E}^0\otimes 1)\circ\widetilde{\theta}=\theta$. Set $\widetilde{\mathbb{H}}\colon E^{M-1}\to E^{M-1}\otimes_\A\A^1$ to be $\nabla_{M-1}+\widetilde{\theta}$. Then $\mathbb{E}^0\widetilde{\mathbb{H}}=\widetilde{\mathbb{H}}\mathbb{E}^0$ in the right most square below. 
\begin{equation*}\label{CD2}
\begin{array}{cccccccc}
 & E^{N} & \mto{\mathbb{E}^0} & E^{N+1} &
\mto{\mathbb{E}^0} \cdots \mto{\mathbb{E}^0}& E^{M-1} & \mto{\mathbb{E}^0} & E^M\\
                       &   &  &  &
            & \nabla_{M-1}\downarrow & \stackrel{\mu}{\searrow} &  \widetilde{\mathbb{H}}\downarrow \\
 & E^{N}\otimes_\A\A^1 &
\mto{\mathbb{E}^0\otimes 1} & E^{N+1}\otimes_\A\A^1 &
\mto{\mathbb{E}^0\otimes 1} \cdots \mto{\mathbb{E}^0\otimes 1} & E^{M-1}\otimes_\A\A^1 & \mto{\mathbb{E}^0\otimes 1} & E^M\otimes_\A\A^1
\end{array}
\end{equation*}
Now we continue backwards to construct all $\widetilde{\mathbb{H}}\colon E^\bullet\to E^\bullet\otimes_\A\A^1$ satisfying $(\mathbb{E}^0\otimes 1)\widetilde{\mathbb{H}}=\widetilde{\mathbb{H}}\mathbb{E}^0=0$. This completes the proof of the lemma.
\end{proof}
\noindent (Proof of the theorem, continued.) Set $\widetilde{\mathbb{E}}^1=(-1)^k \widetilde{\mathbb{H}}$ on $E^k$. Then
\[
\mathbb{E}^0\widetilde{\mathbb{E}}^1+\widetilde{\mathbb{E}}^1\mathbb{E}^0=0,
\] 
but it is not necessarily true that $e^0\widetilde{\mathbb{E}}^1-\mathbb{X}^1 e^0=0$. We correct this as follows. Consider $\psi=e^0\widetilde{\mathbb{E}}^1-\mathbb{X}^1 e^0\colon E^\bullet\to X^\bullet\otimes_\A\A^1$. Check that $\psi$ is $\A$-linear and a map of complexes.
\begin{equation*}\begin{array}{cll}
 & & (E^\bullet\otimes_\A\A^1 , \mathbb{E}^0\otimes 1)\\
 &\stackrel{\widetilde{\psi}}{\nearrow} & \downarrow e^0\otimes 1 \\
 E^\bullet & \mto{\psi} & (X^\bullet\otimes_\A\A^1, \mathbb{X}^0\otimes 1)
 \end{array}
 \end{equation*}
 In the above diagram, $e^0\otimes 1$ is a quasi-isomorphism $e^0$ is a homotopy equivalence. So by Lemma 1.2.5 of \cite{MR506654} there is a lift $\widetilde{\psi}$ of $\psi$ and a homotopy $e^1\colon E^\bullet\to X^{\bullet-1}\otimes_\A\A^1$ between $(e^0\otimes 1)\widetilde{\psi}$ and $\psi$, 
 \[
 \psi-(e^0\otimes 1)\widetilde{\psi}=(e^1\mathbb{E}^0+\mathbb{X}^0 e^1).
  \]
Let $\mathbb{E}^1=\widetilde{\mathbb{E}}^1-\widetilde{\psi}$. Then
\begin{equation*}\label{ide}
\mathbb{E}^0\mathbb{E}^1+\mathbb{E}^1\mathbb{E}^0=0 \mbox{ and } e^0\mathbb{E}^1-\mathbb{X}^1 e^0=e^1\mathbb{E}^0+\mathbb{X}^0 e^1.\end{equation*}

So we have constructed the first two components $\mathbb{E}^0$ and $\mathbb{E}^1$ of the $\mathbb{Z}$-graded connection and the first components $e^0$ and $e^1$ of the quasi-isomorphism $E^\bullet\otimes_\A\A^\bullet \to  X^\bullet\otimes_\A\A^\bullet$. 

To construct the rest, consider the mapping cone $L^\bullet$ of $e^0$. Thus,
\[
L^\bullet=E[1]^\bullet\oplus X^\bullet.
\]
Let $\mathbb{L}^0$ be defined as the matrix
\begin{equation*}\mathbb{L}^0=\left(\begin{array}{cc} \mathbb{E}^0[1] & 0 \\
                                         e^0[1]  & \mathbb{X}^0 \end{array}\right).
\end{equation*}
Define $\mathbb{L}^1$ as the matrix
\begin{equation*}\mathbb{L}^1=\left(\begin{array}{cc} \mathbb{E}^1[1] & 0 \\
                                         e^1[1]  & \mathbb{X}^1 \end{array}\right).
\end{equation*}
Now $\mathbb{L}^0\mathbb{L}^0=0$ and $[\mathbb{L}^0,\mathbb{L}^1]=0$ express the identities \eqref{ide}.
Let 
\begin{equation*}
D=\mathbb{L}^1\mathbb{L}^1+\left(\begin{array}{cc} 0 & 0\\
                                                    \mathbb{X}^2 e^0 & [\mathbb{X}^0,\mathbb{X}^2]
                                                    \end{array}\right).
\end{equation*}
Then, as is easily checked, $D$ is $\A$-linear and  
\begin{enumerate}\item $[\mathbb{L}^0,D]=0$, and
\item $D|_{0\oplus X^\bullet}=0$.
\end{enumerate}
Since $(L^\bullet, \mathbb{L}^0)$ is the mapping cone of a quasi-isomorphism, it is acyclic and since $\A^\bullet$ is flat over $\A$, $(L^\bullet\otimes_\A\A^2, \mathbb{L}^0\otimes 1)$ is acyclic too. Since $E^\bullet$ is projective, we have that
\[
\Hom_\A^\bullet((E^\bullet,\mathbb{E}^0), (L^\bullet\otimes_\A\A^2,\mathbb{L}^0))
\]
is acyclic. Moreover, \[\Hom_\A^\bullet((E^\bullet,\mathbb{E}^0), (L^\bullet\otimes_\A\A^2,\mathbb{L}^0))\subset \Hom_\A^\bullet(L^\bullet,(L^\bullet\otimes_\A\A^2,[\mathbb{L}^0,\cdot]))\] is a subcomplex. Now we have $D\in \Hom_\A^\bullet(E^\bullet, L^\bullet\otimes_\A\A^2)$ is a cycle and so there is $\widetilde{\mathbb{L}}^2\in \Hom_\A^\bullet(E^\bullet, L^\bullet\otimes_\A\A^2)$ such that $-D=[\mathbb{L}^0,\widetilde{\mathbb{L}}^2]$. Define  $\mathbb{L}^2$ on $L^\bullet$ by 
\begin{equation*}
\mathbb{L}^2=\widetilde{\mathbb{L}}^2+\left(\begin{array}{cc} 0 & 0\\
                                                  0 & \mathbb{X}^2
                                                    \end{array}\right). 
\end{equation*}
Then 
\begin{equation*}\begin{split}
[\mathbb{L}^0,\mathbb{L}^2]= & [\mathbb{L}^0,\widetilde{\mathbb{L}}^2+\left(\begin{array}{cc} 0 & 0\\
                                                  0 & \mathbb{X}^2
                                                    \end{array}\right) ] \\
= & -D+[\mathbb{L}^0,\widetilde{\mathbb{L}}^2+\left(\begin{array}{cc} 0 & 0\\
                                                  0 & \mathbb{X}^2
                                                    \end{array}\right) ] \\
= & -\mathbb{L}^1\mathbb{L}^1
\end{split}
\end{equation*}
So 
\[
\mathbb{L}^0\mathbb{L}^2+\mathbb{L}^1\mathbb{L}^1+\mathbb{L}^2\mathbb{L}^0=0.
\]

We continue by setting 
\begin{equation*}
D=\mathbb{L}^1\mathbb{L}^2+\mathbb{L}^2\mathbb{L}^1+\left(\begin{array}{cc} 0 & 0\\
                                                    \mathbb{X}^3 e^0 & [\mathbb{X}^0,\mathbb{X}^3]
                                                    \end{array}\right). 
\end{equation*}
Then $D\colon L^\bullet\to L^\bullet\otimes_\A\A^3$ is $\A$-linear, $D|_{0\oplus X^\bullet}=0,$ and
\[
[\mathbb{L}^0,D]=0.
\]
Hence, by the same reasoning as above, there is $\widetilde{\mathbb{L}}^3\in \Hom_\A^\bullet(E^\bullet, L^\bullet\otimes_\A\A^3)$ such that $-D=[\mathbb{L}^0,\widetilde{\mathbb{L}}^3]$. Define 
\begin{equation*}
\mathbb{L}^3=\widetilde{\mathbb{L}}^3+\left(\begin{array}{cc} 0 & 0\\
                                                  0 & \mathbb{X}^3
                                                    \end{array}\right). 
\end{equation*}
Then one can compute that $\sum_{i=0}^3\mathbb{L}^i\mathbb{L}^{3-i}=0$.

Now suppose we have defined $\mathbb{L}^0,\ldots,\mathbb{L}^n$ satisfying for $k=0,1, \ldots, n,$
\[
\sum_{i=0}^k \mathbb{L}^i\mathbb{L}^{k-i}=0   \hspace{.5in}  \mbox{   for 	} k\ne 2, \\
\]
and 
\[
\sum_{i=0}^2 \mathbb{L}^i\mathbb{L}^{2-i}=0  \hspace{.5in}\mbox{   for   } k=2.
\]
Then define
\begin{equation*}
D=\sum_{i=1}^n \mathbb{L}^i\mathbb{L}^{n+1-i}+\left(\begin{array}{cc} 0 & 0\\
                                                    \mathbb{X}^{n+1} e^0 & [\mathbb{X}^0,\mathbb{X}^{n+1}]
                                                    \end{array}\right). 
\end{equation*}
$D|_{0\oplus X^\bullet}=0,$ and we may continue the inductive construction of $\mathbb{L}$ to finally arrive at a $\mathbb{Z}$-graded connection satisfying $\mathbb{L}\mathbb{L}=0$. The components of $\mathbb{L}$ construct both the $\mathbb{Z}$-graded connection on $E^\bullet$ as well as the morphism from $(E^\bullet,\mathbb{E})$ to $(X^\bullet,\mathbb{X})$. 

It follows from prop. \ref{HE} that $\RH((E^\bullet,\mathbb{E}))\stackrel{e}{\to} (F,f)$ is a quasi-isomorphism.
\end{proof}
\section{Examples and Extensions}\label{sec:5}
\subsection{Riemannian Fibration Example}\label{sec:5.1}
We start with the setup of section $10.1$ in \cite{MR2273508}.  Suppose $\pi\colon M \rightarrow B$ is a fiber bundle with compact fiber, with Riemannian metrics on the fibers.  Assume we have a connection on $M$  ---realized as a splitting,
\begin{equation*}
 TM = T_HM \oplus T(M/B),
\end{equation*}
where $T(M/B)$ is the bundle of vertical tangent vectors, and $T_HM$ a subbundle isomorphic to $\pi^*TB$.  Assume also that we have a connection $\nabla^{M/B}$ on $T(M/B)$. By $d_{M/B}$ we mean the vertical exterior differential on $T(M/B)$, by $P$ the projection
\begin{equation*}
 P\colon TM \rightarrow T(M/B),
\end{equation*}
and by $S$ the second fundamental form (in $T^*(M/B)\otimes T(M/B) \otimes T^*_HM$)
\begin{equation*}
 <S(X,\theta),Z> := <\nabla^{M/B}_Z X - P[Z,X],\theta>,
\end{equation*}
acting on $X$,$\theta$, and $Z$, sections of $T(M/B)$, $T^*(M/B)$, and $T_HM$ respectively.  By $\Omega$, we denote the curvature 
\begin{equation*}
 \Omega(X,Y) := -P[X,Y],
\end{equation*}
a section of $\Hom(\Lambda^2 T^*_HM,T(M/B))$.  

The vertical differential can be extended to $\Gamma(\Lambda T^*(M/B) \otimes \Lambda T^*_HM)$ by 
\begin{equation*}
 d_{M/B}(\beta \otimes \pi^*\nu) = d_{M/B}\beta \otimes \pi^*\nu. 
\end{equation*}
We can use the connection $\nabla^{M/B}$ to define a differential on the same space via
\begin{equation*}
 \delta_B(\beta \otimes \pi^*\nu) = (-1)^{\lvert \beta \rvert}\beta \otimes \pi^*d_B\nu + \sum_\alpha f^\alpha \wedge \nabla^{M/B}_{f_\alpha} \beta \otimes \pi^*\nu, 
\end{equation*}
for any frame $f_\alpha$ in $T(M/B)$ and $f^\alpha$ its dual frame.

Let $A_M^\bullet$ and $A_B^\bullet$ be the deRham differential graded algebras of $M$ and $B$ respectively.  We regard $A^\bullet_M$ as a right $\Z$-graded $A_B^0$-module, with $d_M$ a flat $\Z$-graded connection by virtue of the decomposition (\cite{MR2273508}, prop. 10.1) 
\begin{equation*}
 d_M = d_{M/B} + \delta_B - \sum_i \iota(e_i)<S,e^i> + \sum_i \iota(e_i)<\Omega,e^i>,
\end{equation*}
in which $\{e_i \}$ is a frame in $T(M/B)$ with $e^i$ its dual.  There is no point in describing the last two terms in detail beyond noting that they are elements of the type  
\begin{equation*}
A_{M}^i \otimes_{A_B^0} A_B^j \rightarrow A_M^{i-1}\otimes_{A_B^0} A_B^{j+2}.
\end{equation*}
So letting the latter two terms form $\E^2$, and the first two be $\E^0$ and $\E^1$ respectively, the decomposition of $\E := d_M$ is of the type we have working with already:
\begin{equation*}
 \E = \E^0 + \E^1 + \E^2.
\end{equation*}
 
Therefore, $d_M$ is a flat, $\Z$-graded connection with $\E^0, \E^1,$ and $\E^2$ non-trivial.  Plugging the Hopf fibration into this example might produce one of the simplest examples of a cohesive module with non-trivial higher connection components.

There is an issue that has been suppressed here: this example is infinite dimensional, and hence is only an example of a \emph{quasi-cohesive module} in the parlance of \cite{MR2648899}.  This can be corrected by virtue of some theorems from that article.  

First of all, one can note immediately that $A^\bullet_M$ defines a \emph{quasi-cohesive} $A_B^\bullet$-module (definition 3.2.2 in \cite{MR2648899}).  Furthermore, $d_{M/B}$ is $A^0_B$-nuclear, so therefore by theorem 3.2.7 in \cite{MR2648899}, $A^\bullet_M$ is a \emph{quasi-finite}, quasi-cohesive module.  And hence, the same theorem provides the existence of an actual cohesive module $E=(E^\bullet,\E)$ with the property that $E$ quasi-represents $\tilde{h}_{A_M^\bullet} := \Hom_{A_B^0}(-,A_M^\bullet)$.  In other words, there is a quasi-isomorphism in Mod-$\mathcal{P}_{\mathcal{A}}$ between $\tilde{h}_{A^\bullet_M}$ and $h_E$.  So $E$ is the cohesive module representing this fibration.  

\subsection{Coefficients in some $\mathcal{P}_{B}$}\label{sec:5.3}
As was already mentioned, one can define infinity-local systems valued in any dg-category.  Given such a target category the question arises of what can be said about the correspondence we have proved.  I.e., what dg-category sits on the other side?  Here is one interesting example.  

Suppose $\Sigma$ is some complex manifold, and $B := (\Omega^{0,\bullet}(\Sigma),\bar{\partial})$ its Dolbeaux complex.  In this case, $\mathcal{P}_B$ is a dg-enhancement of the derived category of sheaves with coherent cohomology on $\Sigma$, per \cite{MR2648899}.  Our Riemann--Hilbert correspondence extends to
\begin{equation*}
 \mathcal{P}_{B \hat{\otimes} A(X)} \cong \mathsf{Loc}^{\mathcal{P}_{B}}_{\infty}(X),
\end{equation*}
which is a statement which then concerns infinity-local systems valued in (an enhancement of) the derived category of $B$.

\begin{section}{Acknowledgements}
The authors extend thanks to Tobias Dyckerhoff, Pranav Pandit, Tony Pantev, and Jim Stasheff for helpful comments during the development this work, and especially Kiyoshi Igusa, who gratefully shared his ongoing work on integration of superconnections (arXiv:0912.0249v1).  And lastly the authors are grateful to Camilo Arias Abad and Florian Schaetz for pointing out a significant oversight in our first draft, and for sharing their work (arXiv:1011.4693v2), which is inspired by this one.
\end{section}


\noindent
Jonathan Block\\
University of Pennsylvania\\
David Rittenhouse Laboratory\\
209 S. 33rd St., Philadelphia, PA, 19104\\
Tel.: (215) 898-8178\\
blockj@math.upenn.edu\\

\noindent
Aaron M. Smith\\
University of Waterloo\\ 
Pure Mathematics\\
200 University Avenue West\\
Waterloo, Ontario, N2L 3G1\\ 
Canada\\
aaron.smith@uwaterloo.ca\\
aasmith@alumni.upenn.edu\\

\end{document}